\pdfoutput=1
\RequirePackage{ifpdf}
\ifpdf 
\documentclass[pdftex]{sigma}
\else
\documentclass{sigma}
\fi

\numberwithin{equation}{section}

\newtheorem{Theorem}{Theorem}[section]
\newtheorem*{Theorem*}{Theorem}

\newtheorem{Lemma}[Theorem]{Lemma}

 { \theoremstyle{definition}
\newtheorem{Definition}[Theorem]{Definition}

\newtheorem{Remark}[Theorem]{Remark} }

\usepackage{enumerate, slashed}

\begin{document}
\allowdisplaybreaks

\renewcommand{\thefootnote}{}

\newcommand{\arXivNumber}{2306.00590}

\renewcommand{\PaperNumber}{102}

\FirstPageHeading

\ShortArticleName{A Note on the Spectrum of Magnetic Dirac Operators}

\ArticleName{A Note on the Spectrum of Magnetic Dirac Operators\footnote{This paper is a~contribution to the Special Issue on Global Analysis on Manifolds in honor of Christian B\"ar for his 60th birthday. The~full collection is available at \href{https://www.emis.de/journals/SIGMA/Baer.html}{https://www.emis.de/journals/SIGMA/Baer.html}}}

\Author{Nelia CHARALAMBOUS~$^{\rm a}$ and Nadine GROSSE~$^{\rm b}$}

\AuthorNameForHeading{N.~Charalambous and N.~Gro{\ss}e}

\Address{$^{\rm a)}$~Department of Mathematics and Statistics, University of Cyprus, Nicosia, 1678, Cyprus}
\EmailD{\href{mailto:charalambous.nelia@ucy.ac.cy}{charalambous.nelia@ucy.ac.cy}}

\Address{$^{\rm b)}$~Mathematisches Institut, Universit\"at Freiburg, 79100 Freiburg, Germany}
\EmailD{\href{mailto:nadine.grosse@math.uni-freiburg.de}{nadine.grosse@math.uni-freiburg.de}}

\ArticleDates{Received June 02, 2023, in final form December 14, 2023; Published online December 22, 2023}

\Abstract{In this article, we study the spectrum of the magnetic Dirac operator, and the magnetic Dirac operator with potential over complete Riemannian manifolds. We find sufficient conditions on the potentials as well as the manifold so that the spectrum is either maximal, or discrete. We also show that magnetic Dirac operators can have a dense set of eigenvalues.}


\Classification{58J50; 35P05; 53C27}

\renewcommand{\thefootnote}{\arabic{footnote}}
\setcounter{footnote}{0}

\section{Introduction}

In this article, we consider the magnetic Dirac operator, and the magnetic Dirac operator with potential over a complete Riemannian manifold with an associated Clifford bundle. Our main goal is to study the spectral properties of these operators depending on the behavior of both the magnetic potential and the additional potentials. In particular, we are interested in finding sufficient conditions on the potentials so that the essential spectrum of the operator is either maximal, in other words ${\mathbb R}$, or discrete.

The spectral properties of magnetic Schr\"odinger operators, over $\mathbb{R}^2$ and $\mathbb{R}^3$ have been extensively studied, see, for example, \cite{CFKS, Helffer, shi} and references therein. In this case, the operators are classical Schr\"odinger operators (the Laplacian with a scalar-valued potential) plus a magnetic field which is a vector field acting via some Clifford-type action on $\mathbb C^2$-valued $L^2$-integrable functions over the Euclidean space. In \cite{CFKS}, Cycon, Froese, Kirsch and Simon show that whenever the potential is relatively compact with respect to the
Laplace operator and the magnetic field vanishes at infinity, the spectrum of the magnetic Schr\"odinger operator coincides with the spectrum of the Laplacian and is $[0,\infty)$, it is in other words maximal \cite[Theorem~6.1]{CFKS}. On the other hand, Miller and Simon show that depending on the decay rate of the magnetic field, the Hamiltonian operator of a spinless particle can have purely absolutely continuous spectrum, dense point spectrum in $[0,\infty)$, or dense point spectrum in a closed interval and absolutely continuous spectrum in its complement \cite{MS} (see also \cite[Theorem~6.2]{CFKS}). The case of `pure' Schr\"odinger operators over Euclidean spaces was extensively studied (see, for example, \cite{Agmon, simon}).

The case of the Dirac operator with varying types of potentials over Euclidean spaces was considered in \cite{HNW, Yamada}. General notes on the occurrence and meaning of different type of potentials for the Dirac operator in ${\mathbb R}^{1,3}$ and discussions on self-adjointness can be found in \cite[Sections~4.2 and~4.3]{Thaller}, see also Remark~\ref{rem_pot_types}. Yamada showed for example in~\cite{Yamada} that if the mass-type potential becomes unbounded at infinity and its derivative as well as the one of the electric potential do not grow faster than the potential itself, then the spectrum will be purely discrete.

While several results for Schr\"odinger operator both with scalar potential and with magnetic fields were transferred to manifolds, see, e.g., \cite{shi2, Shubin}, to the best of our knowledge little has been done for magnetic Dirac operators over complete Riemannian manifolds and their spectrum. Over compact odd-dimensional spin manifolds, with certain restrictions on the eigenvalues of the contact endomorphism, Savale proves Weyl-type of estimates for the number of eigenvalues of the coupled Dirac operator when the magnetic form is a contact 1-form, and he also provides a limit formula for the eta invariant of the magnetic Dirac operator when the magnetic potential is a non-resonant contact form in \cite{sav, sav2}. In the context of compact manifolds with boundary, spectral estimates for Callias-type Dirac operators that can be interpreted as Dirac operators with mass-type potentials were studied for example by Cecchini and Zeidler in \cite{CZ}.

In this short note, we start by generalizing the small selection of the theorems in~\cite{CFKS} and~\cite{Yamada} which we mentioned above, to a more general class of complete Riemannian manifolds with Dirac-type operators.

First, we give some preliminaries on the different types of potentials for Dirac operators. We restrict ourselves to potentials that are functions or one-forms acting via Clifford multiplication as in Remark~\ref{rem_diff_conn}. Then, in Theorem~\ref{thm_pot} we prove that the spectrum of the magnetic Dirac operator is discrete if the mass-type potential goes to infinity at infinity and dominates the magnetic and electric potential as well as the derivatives of all these potentials. On the other hand, in Theorem~\ref{thm_mag_zero} we show that a pure magnetic Schr\"odinger operator over asymptotically flat manifolds with magnetic field vanishing at infinity has essential spectrum $[0,\infty)$. Moreover, we will see in Theorem~\ref{thm_dense} that similarly to magnetic Schr\"odinger operators, magnetic Dirac operators can also have dense eigenvalues. Our results illustrate that the vast variability in the behavior of the spectrum of Schr\"odinger operators, can be generalized to magnetic Dirac operators with potential over complete manifolds. Certainly, there are many intermediate cases depending on the behaviour of the various potentials at infinity which are not discussed here, and which could merit further investigation.

\section{Preliminaries}
We consider a complete Riemannian manifold $(M^m,g)$. Let $S\to M$ be a Clifford bundle. That means $S$ is a bundle of Clifford modules over $M$ which has a fiberwise Hermitian metric $\langle\cdot,\cdot\rangle$ and a metric connection $\nabla^S$ on $S$ such that the Clifford action of a tangent vector is skew-adjoint with respect to $\langle\cdot,\cdot\rangle$, and $\nabla^S$ is compatible with the Levi-Civita connection on $M$, cf.\ \cite[Definition~3.4]{Roe}. The sections of $S$ are called spinors and the space of smooth section will be denoted by $\Gamma(S)$.	
The Riemannian metric induces an $L^2$-inner product on spinors over $M^n$ by
\[( \varphi, \eta ) := \int_M \langle\varphi, \eta\rangle\, {\rm d}v,\]
where ${\rm d}v$ is the Riemannian measure.

Moreover, there is an associated Dirac operator acting on the sections of $S$ which we denote by $D$. When $M$ is a spin manifold and $S$ is the associated spinor bundle, the associated Dirac operator is known as the \emph{classical Dirac operator}. It is well known that the square Dirac operator satisfies the Weitzenb\"ock formula
\begin{align*}
D^2=\nabla^*\nabla +\mathcal{R},
\end{align*}
where $\mathcal{R}\in \operatorname{End}(S)$ is the \emph{Clifford contraction} which acts as a tensor on spinors. When $M$ is spin and $D$ is the classical Dirac operator, the Clifford contraction is simply a constant multiple of the scalar curvature of the manifold, $\mathcal{R}=\mathrm{scal}/4$. The Weitzenb\"ock formula allows us to treat the square Dirac operator as a Schr\"odinger-type of operator, with the Clifford contraction as its potential. Controlling this potential allows us to then obtain analytical and spectral properties for the operator. When one considers the classical Dirac operator, this process reduces to controlling the scalar curvature of the manifold.

Next we present the different types of potentials. We introduce the notation step by step, since for magnetic and electric potentials less structure is needed than for mass-type potentials.

\subsection{Magnetic Dirac operators}
Let $(M^m,g)$ be a Riemannian manifold with Clifford bundle $S$ and associated Dirac operator~$D$. Let $A\in \Omega^1(M)$ be a smooth real-valued one-form on $M$.

The \emph{magnetic Dirac operator} with \emph{magnetic potential} $A$ is defined as
\[ D_{A}:= D+ {\mathrm{i}}A\cdot \colon \ \Gamma(S)\to \Gamma(S) \]
and the \emph{spinorial magnetic connection Laplacian} as
\[ H_{A} := \bigl(\nabla^S + {\mathrm{i}}A\bigr)^*\bigl(\nabla^S + {\mathrm{i}}A \bigr) \colon \ \Gamma(S)\to \Gamma(S),\]
where $A\colon \phi\in \Gamma(S)\mapsto A\otimes \phi \in \Gamma(T^*M\otimes S)$ and $A\cdot \colon \phi \in \Gamma(S) \mapsto A\cdot \phi \in \Gamma(S)$ is given by the Clifford multiplication.

\begin{Remark}\label{rem_diff_conn} We note that adding the ${\mathrm{i}}A$-term is a special type of change of connection for the bundle: In general, if $\nabla^1$ and $\nabla^2$ are two connections on the same bundle $S$, then $\nabla^1-\nabla^2=\omega\in \Omega^1 (M, \operatorname{End} S)$. Adding ${\mathrm{i}}A$ corresponds to the special choice $\omega s:= {\mathrm{i}} A\otimes s$ for all $s\in \Gamma(S)$. Since $\langle {\mathrm{i}}A(X)\phi,\psi\rangle=-\langle\phi, {\mathrm{i}}A(X) \psi\rangle$ for all $p\in M$, $X\in T_p M$ and $\phi, \psi\in S_p$, the connection $\nabla +{\mathrm{i}}A$ is also metric. Moreover, it is still compatible with the Clifford action.
The spinc construction is related to the above setting. In that case $A$ is the connection on an auxiliary line bundle which is tensored to the (locally) spinor bundle. In principle, we could also generalize the potentials used here and in the following by using bundles constructed from twisting the spinor bundle. But here we restrict to the case of no additional bundle.\looseness=-1\end{Remark}

The operators $D_A$ and $H_A$ are essentially self-adjoint on complete manifolds~\cite[pp.~411--412]{Chernoff}.

Next we calculate the Lichnerowicz formula (see, for example, \cite{LM,Roe}) for the explicit choice of the Dirac operator which we will be using in this article.

\begin{Lemma}\label{lem_DAsquare} It is
\[ D_A^*D_A = H_A + \mathcal{R} + {\mathrm{i}} \, {\rm d}A\cdot\qquad \text{and}\qquad H_{A} = \bigl(\nabla^S\bigr)^*\nabla^S -2{\mathrm{i}}  \nabla_A^S -{\mathrm{i}}  \operatorname{div} A +|A|^2, \]
where $\mathcal{R}\in \operatorname{End}(S)$ is the Clifford contraction and where for a local orthonormal frame $e_i$ of $TM$ the action of ${\rm d}A\cdot$ is given by \[{\rm d}A\cdot \phi := \sum_{i<j} {\rm d}A(e_i,e_j)e_i\cdot e_j\cdot \phi.\]
\end{Lemma}

We note that $H_A$ is not of standard Schr\"odinger type, since a first derivative $\nabla_A^S$ appears.
\begin{proof}We perform a local calculation at a given point $p\in M$. In a neighbourhood of $p$, we choose a local orthonormal frame $\{e_i\}_i$ such that $\nabla_{e_i}e_j(p)=0$ for all $i$, $j$.
Let $A=\sum_i A_i e_i^\flat$, where $\bigl\{e_i^\flat\bigr\}_i$ denotes the dual cotangent frame. Using the fact that $D_A^*=D_A$, we compute, pointwise,
\samepage{\begin{align*}
D_A^*D_A \phi &= (D+{\mathrm{i}}A\cdot )(D+{\mathrm{i}}A\cdot) \phi \\[1.5ex]
&= D^*D \phi +{\mathrm{i}}  e_i\cdot \nabla^S_{e_i} (A_j e_j\cdot \phi ) +{\mathrm{i}}A\cdot D\phi - \sum_{i,j} A_iA_j e_i\cdot e_j \cdot \phi \\
&= \bigl(\nabla^S\bigr)^*\nabla^S\phi+\mathcal{R}\phi +{\mathrm{i}}\,  {\rm d}A\cdot \phi -{\mathrm{i}}  \operatorname{div} A  \phi - 2{\mathrm{i}}\sum_{i} A_i \nabla^S_{e_i}\phi +|A|^2\phi.
\end{align*}}	
Note, that $A^*(e_i\otimes \phi)= A_i\phi$ and $\bigl(\nabla^S\bigr)^*(e_i\otimes \phi)=-\nabla^S_{e_i}\phi$. Hence,
\begin{align*}
H_{A}\phi = \bigl(\nabla^S + {\mathrm{i}}A\bigr)^*\bigl(\nabla^S + {\mathrm{i}}A \bigr) \phi
= \bigl(\nabla^S\bigr)^*\nabla^S\phi -2{\mathrm{i}}\sum_i A_i \nabla^S_{e_i}\phi -{\mathrm{i}}\operatorname{div} A \phi +|A|^2\phi. \tag*{\qed}
\end{align*}\renewcommand{\qed}{}
\end{proof}

\subsection{With electric potential}
Instead of only working with magnetic potentials we can also add an electric potential $A_0$, in other words a function on $M$ and consider the more general operator
\[ D_{A, A_0}=D_A+A_0.\]
\begin{Remark}\label{rem_pot_types}
Above $A$ is interpreted as the magnetic potential and $A_0$ as the electric potential which can be seen as follows: If we work on ${\mathbb R}^{1,3}$, with coordinates $x^0, \ldots, x^3$ the (physical) Dirac operator with potential $A=A_\mu {\rm d}x^\mu$ looks like ${\mathrm{i}}\gamma_\mu (\partial_\mu + {\mathrm{i}}A_\mu)$ (by borrowing the notation from physics using gamma-matrices, cf.\ \cite{Thaller}). A spinor in its kernel satisfies ${\mathrm{i}}\gamma_0 \partial_t \phi = \bigl(-{\mathrm{i}} \vec{\gamma}\vec{\nabla} + \vec{\gamma} \vec{A} + \gamma_0A_0\bigr)\phi$, where the right hand-side is $D_{A,A_0}\phi$ as defined above. If $A$ itself does not depend on time, $A_0$ gives the part of the electric field in $F={\rm d}A$ and $(A_1, A_2, A_3)$ the magnetic field (also compare to \cite[Section~4.2]{Thaller}).
\end{Remark}

\begin{Remark}\label{rem:ess_sa} Let $V$ be a zero order symmetric operator on $S$ acting as $(V\phi)(x)=V(x)\phi(x)$ with $V(x)\in \mathrm{Hom} (S_x,S_x)$. Assume that for each point $p\in M$ there is a $\tilde{V}_p\in L^2_{\rm{loc}}$ that coincides with $V$ in a neighbourhood of $p$ and such that $D_A+V$ is essentially self-adjoint on $C_c^\infty(S)$. Then by \cite[Theorem~2.1]{Chernoff2}, $D_{A,V}$ is essentially self-adjoint.

In particular, this is true for $V$ smooth~\cite[Theorem~2.2]{Chernoff}.
In case the potential is singular, the essential self-adjointness depends both on how fast the singularity blows up as well as the coupling constant, in other words the scalar multiple in front of it (see \cite[Section~4.3.3]{Thaller}). Here we will only consider $C^1$-potentials for simplicity.
\end{Remark}

\subsection{With (mass-type) potential}\label{sec:potential}
Let again $(M^m,g)$ be a Riemannian manifold with Clifford bundle $S$ and associated Dirac operator $D$.

Assume that $S$ is $\mathbb Z_2$-graded, in other words there is a parallel and orthogonal $\mathbb Z_2$-grading $S=S^+\oplus S^-$ and the Clifford multiplication by a tangent vector is an odd map with respect to this grading. Then the associated Dirac operator has the form
\[ D= \begin{pmatrix} 0 & D^-\\ D^+ & 0
\end{pmatrix}\]
with respect to this splitting. Moreover, we define $\nu\cdot (\phi,\psi)= {\mathrm{i}}  (\phi, -\psi)$. This corresponds to a choice of representation of the Clifford action as in \cite{Ginoux,Thaller}. Then, $\nu\cdot$ commutes with Clifford multiplication for all $X\in \Gamma(TM)$, since it is an odd map. Then we define for $V\in L^2_{\rm{loc}}$,
\[ D_{A,V, A_0}:= D+ \mathrm{i} A\cdot + {\mathrm{i}}V \nu\cdot + A_0\colon \ \Gamma\bigl(\tilde{S}\bigr)\to \Gamma\bigl(\tilde{S}\bigr).\]
If $A_0=0$, we write $D_{A,V}:= D_{A,V,0}$. If $A=0$ and $A_0=0$, we write $D_V:= D_{0,V,0}$.
\begin{Remark}
Let $M$ be spin (with a chosen spin structure) and $S$ the spinor bundle. In case $\dim M$ is even, the spinor bundle is automatically $\mathbb Z_2$-graded. Then $\iota^*S$, $\iota\colon M\hookrightarrow M\times {\mathbb R}$, is the spinor bundle of $M\times {\mathbb R}$ and the $\nu\cdot $ from above is the Clifford multiplication by the unit vector in the ${\mathbb R}$ direction. In case $\dim M$ is odd, then $\tilde{S}:= S\oplus S$ and $X\cdot (\phi, \psi) = (X\cdot \phi, -X\cdot \psi)$ for $X\in \Gamma(TM)$ defines a $\mathbb Z_2$-graded Clifford bundle as above, cf.\ \cite[Section~1.4]{Ginoux}.
\end{Remark} 	
For ${\mathbb R}^{1,3}$, in other words for the manifold $M={\mathbb R}^3$ with ${\mathbb C}^4$-valued spinors, the $V$-term above corresponds to the scalar potential in \cite[Section~4.2.1]{Thaller}.
Over Riemannian manifolds the operator $D+{\mathrm{i}}V \nu\cdot$ is often referred to as a Callias-type operator, see \cite{CZ}.

For the essential self-adjointness Remark~\ref{rem:ess_sa} is applicable.
\begin{Lemma} \label{DAV_square}For $V\in C^1(M)$, we have $D_{A,V}^*D_{A,V}= D_A^*D_A + {\mathrm{i}}\,  {\rm d}V\cdot \nu\cdot +V^2$.
\end{Lemma}
\begin{proof}It is
\begin{align*}
D_{A,V}^*D_{A,V} &=D_A^*D_A - V A\cdot \nu\cdot -V\nu\cdot A\cdot +{\mathrm{i}}V\nu\cdot D +{\mathrm{i}}  D(V\nu\cdot ) +V^2 \\
& =D_A^*D_A + {\mathrm{i}}\,  {\rm d}V\cdot \nu\cdot +V^2,
\end{align*}
where we use the anticommutativity between $D$ and Clifford multiplication by $\nu$ in the last\linebreak line.\end{proof}

\begin{Remark} If $V={\rm const}$, then ${\mathrm{i}}V\nu \cdot$ can be interpreted as a mass term and $D_{A,V}$ as a massive (magnetic) Dirac operator. In this case, the term ${\rm d}V$ in $D^*_{A,V}D_{A,V}$ vanishes and the operator becomes a truly Schr\"odinger-type operator. Thus, in general, such a potential can be interpreted as a position-dependent mass term.\end{Remark}

\subsection{Diamagnetic inequality}
For a metric connection $\nabla$ on a vector bundle, we have $|{\rm d}|\phi||\leq |\nabla \phi|$ for all $\phi\in C^\infty(S)$ and at all points where $\phi\neq 0$.
Since $\nabla +{\mathrm{i}}A$ is just another metric connection on the bundle, the following diamagnetic inequality holds:
\[ |{\rm d}|\phi||\leq |(\nabla + {\mathrm{i}}A )\phi|.\]

\section{Spectral results}
In this section, we will explore whether and how the results of \cite[Section~6]{CFKS} and \cite{Yamada} generalize to complete manifolds under appropriate conditions.

\subsection{Potentials and discrete spectrum}
It is known that the self-adjoint Schr\"odinger operator ${H}_V:= \bigl(\nabla^S\bigr)^*\nabla^S + V^2\colon \operatorname{dom}  H_V\subset L^2(S)\to L^2(S)$ has discrete spectrum whenever the potential term blows up at infinity (condition $(i)$ in Theorem \ref{thm_pot}) since the quadratic form associated to the operator can be made as large as possible on the complement of a large enough ball (see, for example, \cite[Theorem~XIII.16]{RS4}). In~this section we show that this result also holds for the magnetic Dirac operator with potential under the additional assumption that the magnetic potential is bounded and that the derivative of the magnetic potential and the potential $A_0$ is dominated by~$V$.

We note that having a magnetic and/or electric potential that goes to infinity at infinity is not enough to ensure that the spectrum is discrete.

In the following let $|x|:= d(x,p)$ for a fixed $p\in M$.

\begin{Theorem}\label{thm_pot} Assume that $M$ is complete, connected and noncompact. Let $S$ be a $\mathbb Z_2$-graded Clifford bundle as in Section~$\ref{sec:potential}$. Let $V, A_0\in C^1(M)$ and $A\in \Omega^1(M)$.
Assume
\begin{enumerate}[$(i)$]\itemsep=0pt
		\item $\lim_{|x|\to \infty} |{V}(x)|=\infty$;
		\item $\mathcal{R}$ and $A$ are both bounded;
		\item $|{\rm d}V| + |{\rm d}A| + |{\rm d}A_0|+ |\operatorname{div} A|= \mathcal{O}(V)$ as $|x|\to \infty$;
		\item $|A_0|\leq \epsilon |V|$ on $|x|\geq R_0$ for an $\epsilon <1$ and some $R_0>0$.
\end{enumerate}
Then the spectrum of $D_{A,V, A_0}= D +{\mathrm{i}}A\cdot + {\mathrm{i}}V\nu \cdot +A_0$ is discrete.
\end{Theorem}

This generalizes a result for $M={\mathbb R}^3$ with $D={\slashed{D}}$ by \cite{Yamada}, and our proof follows the same idea as in this paper. In \cite{Yamada}, only an electric potential and a mass-type potential are considered.

\begin{proof} In order to prove the theorem, we will show that $(D_{A,V, A_0}-\mathrm{i})^{-1}$ is compact. For that let $\phi_n\in L^2(S)$ be a bounded sequence and set $\psi_n:= (D_{A,V, A_0}-\mathrm{i})^{-1}\phi_n \in H^1_{\rm loc}(S)$.
Since $D_{A,V, A_0}$ is essentially self-adjoint, its spectrum is real and hence $(D_{A,V, A_0}-\mathrm{i})^{-1}$ is bounded. Thus, $\psi_n$ is a bounded sequence in $L^2$. We need to show that this sequence has a convergent subsequence.
	
As mentioned above ${H}_V:= \bigl(\nabla^S\bigr)^*\nabla^S + V^2-A_0^2$ has discrete spectrum. By \cite[Theorem~XIII.64]{RS4}, the set
\[S_b=\bigl\{\phi\in \operatorname{dom}  {H}_V  \subset  L^2(S)\mid \Vert \phi\Vert_{L^2}\leq 1, \bigl(\phi, \bigl(\bigl(\nabla^S\bigr)^*\nabla^S + {V}^2-A_0^2\bigr)\phi\bigr)\leq b \bigr\}\]
is compact for all $b$. We want to show that the sequence $\psi_n$ is in a set $S_b$ for some $b$, then our claim follows by the compactness of this set.
		
We apply $D_{A,V} -A_0 + \mathrm{i}= D_{A,V,-A_0} + \mathrm{i}$ on both sides of $D_{A,V,A_0} \psi_n - {\mathrm{i}}  \psi_n =\phi_n$, and use Lemmas~\ref{DAV_square} and~\ref{lem_DAsquare}, as well as the self-adjointness of $D_{A,V}$ to obtain
\begin{align*}
& \bigl( \bigl(\nabla^S\bigr)^* \nabla^S -2{\mathrm{i}}  \nabla^S_A + |A|^2 + \mathcal{R} -\mathrm{i}  \operatorname{div} A +{\mathrm{i}}\, {\rm d}A\cdot \\
& \qquad {}+ 1+ {\mathrm{i}} \, {\rm d}V\cdot \nu\cdot +V^2 - A_0^2 +2A_0\mathrm{i}+ {\rm d}A_0\cdot \bigr)\psi_n= (D_{A,V, -A_0} +\mathrm{i})\phi_n.\end{align*}
Define
$\mathcal{H} := \bigl\{ \phi \in L^2(S)\mid \Vert \phi\Vert_{{V}}^2:= \Vert \phi\Vert_{H^1}^2 + \bigl\Vert \sqrt{{V^2}-A_0^2}\phi\bigr\Vert_{L^2}^2<\infty \bigr\}$. Then $\mathcal{H}\subset H^1(S)$ is a Hilbert space with inner product $(\phi, \psi)_{{V}}:= (\phi, \psi)_{H^1}+\bigl(({V}^2-A_0^2\bigr)\phi, \psi)_{L^2}$. By $(iv)$, $V^2\leq a\bigl(V^2-A_0^2\bigr)\leq a V^2$ for some $a>1$. Thus, the norms $\Vert \phi\Vert_{{V}}$ and $(\Vert \phi\Vert_{H^1}^2 + |V| \Vert\phi\Vert_{L^2}^2)^\frac{1}{2}$ are equivalent.
	
By assumptions $(iii)$ and $(iv)$, there exist $R>R_0$ and $C>0$ such that $(|{\rm d}V| + |{\rm d}A| + |\operatorname{div} A|+|{\rm d}A_0| +|A_0|)(x)\leq C {V}(x)$ for all $|x|\geq R$. Let $\eta$ be a smooth cut-off function with $\eta(x)=1$ for $|x|\geq R+1$ and $\eta(x)=0$ for $|x|\leq R$. Let $u\in C_c^\infty(S)$. Then	
\begin{align*}
(\eta \psi_n, u)_{V} =&\bigl(\bigl(\nabla^S\bigr)^* \nabla^S (\eta \psi_n), u\bigr)_{L^2}+ \bigl(\bigl(1+V^2-A_0^2\bigr)\eta\psi_n, u\bigr)_{L^2}\\	
=&\bigl(\eta \bigl(\nabla^S\bigr)^* \nabla^S \psi_n, u\bigr)_{L^2}+ \bigl(\bigl(1+V^2\bigr)\eta\psi_n, \psi\bigr)_{L^2}+ \bigl(\psi_n  \Delta \eta -2 \nabla^S_{\nabla\eta} \psi_n, u\bigr)_{L^2}\\
=&\bigl(\phi_n, (D_{A,V,-A_0} -\mathrm{i})(\eta u)\bigr)_{L^2}+ \bigl(\psi_n \Delta \eta - 2 \nabla^S_{\nabla \eta} \psi_n,u\bigr)_{L^2}\\
&+\bigl(\bigl(2{\mathrm{i}}  \nabla^S_A -  |A|^2 -  \mathcal{R} + \mathrm{i}  \operatorname{div} A -  {\mathrm{i}} \, {\rm d}A\cdot -2A_0\mathrm{i}  -{\rm d}A_0\cdot - {\mathrm{i}} \, {\rm d}V \cdot  \nu\cdot \bigr) \psi_n, \eta u\bigr)_{L^2}\\
=&\bigl(\phi_n, (D_{A,V,A_0} -\mathrm{i})(\eta u)\bigr)_{L^2}+ \bigl( \psi_n ,u  \Delta \eta + 2 \nabla^S_{\nabla \eta} u\bigr)_{L^2}\\
&+\bigl(\psi_n, (2{\mathrm{i}}  \nabla^S_A - \mathrm{i}  \operatorname{div} A-  {\mathrm{i}}\,  {\rm d}A \cdot +2A_0\mathrm{i}  + {\rm d}A_0\cdot - {\mathrm{i}}  \, {\rm d}V \cdot  \nu \cdot  - \mathcal{R} -  |A|^2) (\eta u)\bigr)_{L^2}.
\end{align*}
	
Using our assumptions together with $\bigl| \nabla^S_X u\bigr|\leq |X|\bigl| \nabla^S u\bigr|$, we obtain
\begin{align*}
|(\eta \psi_n, u)_{V}| \lesssim \Vert \phi_n\Vert_{L^2}\Vert (D_{A,V,-A_0} -\mathrm{i})(\eta u)\Vert_{L^2} + \Vert \psi_n\Vert_{L^2} \Vert \eta u \Vert_{V}^2
\end{align*}
and, as a result
\begin{align*}
(\eta \psi_n, u)_{V} \leq C(\Vert \phi_n\Vert +\Vert \psi_n\Vert )\Vert u\Vert_{V}\leq C' \Vert u\Vert_{V}.
\end{align*}
Thus, $\Vert \eta \psi_n\Vert_{V}$ is bounded for all $n$. Hence, this sequence lies in the compact subset $S_b$ for some~$b$ and therefore has a convergent subsequence in $L^2(S)$.
\end{proof}

\subsection{Magnetic Schr\"odinger operator with vanishing magnetic field at infinity}

In this section, we prove that the essential spectrum of a magnetic Schr\"odinger operator with vanishing magnetic field at infinity will be maximal over asymptotically flat manifolds. Our result generalizes \cite[Theorem~6.1]{CFKS} who considered the case of Euclidean space. The main idea is to compare the spectrum of the magnetic Schr\"odinger operator to that of certain gauge perturbed operators, using the fact that the magnetic field vanishes at infinity. Since the spectrum of the Laplacian on functions is maximal over asymptotically flat manifolds, and we have `nice' spinors over this space, we can then find an appropriate family of approximate eigenspinors to prove that the essential spectrum of our operator is also maximal.

\begin{Definition}
A manifold $(M^m,g)$ with $m\geq 2$ is asymptotically flat of order $\tau$ if there is a~compact set $K\subset M$ and a diffeomorphism
$\varphi\colon M \setminus K \to \mathbb{R}^m \setminus B_{r_o}(0)$ such that
\begin{gather*}
(\varphi_* g)_{ij} (x) = \delta_{ij} + O\bigl(|x|^{-\tau}\bigr) \qquad \text{and}\\
\partial_k (\varphi_* g)_{ij} (x) = O\bigr(|x|^{-(\tau+1)}\bigr), \qquad \partial_{kl} (\varphi_* g)_{ij} (x)= O\bigl(|x|^{-(\tau+2)}\bigr)
\end{gather*}
for some $\tau>0$, where $|x|$ is the Euclidean distance of $x$ to $0$.
\end{Definition}

Note that under the above assumption on the metric, the curvature tensor of the manifold and hence its scalar curvature tend to zero as $|x|\to \infty$.

To prove that the spectrum of the magnetic Schr\"odinger operator is maximal requires, as we will see, the construction of a large class of $L^2$ integrable approximate eigenspinors. Even though the spectrum of the Laplacian on functions would be $[0,\infty)$ over these manifolds \cite{ChLu2}, we cannot use the bounded test functions with $L^1$ estimates constructed in that article. Instead, we use the idea of \cite{CFKS} choosing gauges such that the magnetic potential becomes small on growing balls, as well as the volume doubling property which holds for asymptotically flat manifolds.
This volume doubling property is defined as follows: There is an $R_o>r_o$ such that for all $R>0$ and $p\in M\setminus \varphi^*\bigl({\mathbb R}^n\setminus B_{3R+R_o}(0)\bigr)$,
\begin{equation} \label{VD}
\operatorname{vol}  B_{2R}(p)\leq C \operatorname{vol}  B_{R}(p),
\end{equation}
where $C$ is a uniform constant, independent of $p$ and $R$. For asymptotically flat manifolds this follows completely analogously to the estimates in \cite[Proof of Proposition~1]{Paeng}. The only difference between our volume doubling property and the statement in \cite{Paeng} is that the balls in the latter are all centered at zero.

\begin{Theorem}\label{thm_mag_zero}
Let $(M^m,g)$ be an asymptotically flat spin Riemannian manifold of order $\tau>0$, with Dirac operator $D$. Let $A\in \Omega^1(M)$ such that $|B={\rm d}A|\to 0$ as $|x|\to \infty$, in other words the magnetic field vanishes at infinity. 	
Then, $[0,\infty) = \sigma_{\mathrm{ess}} (D_A^*D_A)$. In particular, ${\mathbb R} = \sigma_{\mathrm{ess}} (D_A)$ in dimension $0$, $1$, $2$ $\rm {mod}$ $4$.
\end{Theorem}

\begin{proof}First, we will argue that $W=\frac{\mathrm{scal}}{4}+{\mathrm{i}}B\cdot $ is relatively compact to $H_A$, in other words $W(H_A+1)^{-1}$ is a compact operator on $L^2$. The argument is mainly analogous to \cite[Example~6, p.~117 and Problem 41, p.~369]{RS4} but will be given here for completeness. We approximate $W=\frac{\mathrm{scal}}{4}+{\mathrm{i}}B\cdot\colon L^2\to L^2$, by $W_n = \chi_{B_n(p)} W$ where $\chi_{B_n(p)}$ is the characteristic function of the geodesic ball $B_n(p)\subset M$. Since the scalar curvature and the magnetic field $B$ tend to zero at infinity, $\lim_{n\to \infty} \|W- W_n\| \to 0$ in the operator norm. Let now $\phi_k$ be a bounded sequence in $\operatorname{dom} H_A=\{\phi\in L^2\mid H_A\phi\in L^2\}$ equipped with the graph norm. Then, $\phi_k$ is locally also bounded in~$H^1$ and hence locally strongly converges in~$L^2$ to some $\phi$ (after possibly taking a~subsequence). In particular, $W_n\phi_k$ converges to $W_n \phi=\chi_{B_n(p)} W \phi$ in $L^2$. This implies that~$W_n$ is compact as an operator from $\operatorname{dom}H_A$ (equipped with the graph norm) to $L^2$ and hence, that $W_n(H_A+1)^{-1}$ is compact. Since $\lim_{n\to \infty} \|W- W_n\| \to 0$, then also $W(H_A+1)^{-1}$ is compact. By \cite[Corollary 2, p.~113]{RS4}, we get $\sigma_{\text{ess}} (D_A^*D_A) = \sigma_{\text{ess}} (H_A)$.
	
Next, we want to show that $[0, \infty)\subset \sigma_{\text{ess}} (H_A)$. As in \cite[Step~3, p.~117]{CFKS}, we want to use the following fact. Let $T\geq 0$ be a self-adjoint operator on a Hilbert space $H$. Whenever there exists an orthonormal sequence $\psi_n$ in $H$ which converges weakly to zero and satisfies $\bigl\Vert (T+1)^{-1}(T-\lambda)\psi_n\bigr\Vert\to 0$ then $\lambda\in \sigma_{\text{ess}}(T)$.
		
Before we find the test functions $\psi_n$, we first exploit the condition that the magnetic field vanishes at infinity.
	
For any $p \in M$ we define $|p|=|\varphi(p)|$. Since $|B|\to 0$, there is a sequence of points $q_n\in M$ with $|q_n|\to \infty$ and $\sup_{x\in B_{2n}(q_n)} |B|\leq \frac{1}{n^2}$. By taking a subsequence, we may choose $q_n$ such that $|q_n|> 6n+R_o$.
	
Since the spinor bundle is trivial on $M\setminus K$, we can choose a spinor $\Phi$ that is constant with norm $1$ on $M\setminus K$. We note that by our assumptions on the decay of the metric, and hence for the Christoffel symbols corresponding to our coordinates, we get
\[\bigl| {\nabla}^S \Phi\bigr|(x)=\mathcal{O}\bigl(|x|^{-(\tau+1)}\bigr)\qquad \text{and}\qquad\bigl|\bigl(\nabla^S\bigr)^*\nabla^S \Phi\bigr|(x)= \mathcal{O}\bigl(|x|^{-(\tau+2)}\bigr)\] for $|x|\to \infty$ (see \cite[Section~4]{LM} for the spinorial connection $\nabla^S$ in local coordinates).
	
Now let $g_n\colon M\to [0,1]$ be a radial (with respect to $q_n$) smooth cut-off function which is constant on $B_{n}(q_n)$, $0$ on $M\setminus B_{2n}(q_n)$ and satisfies $| \nabla g_n| \leq \frac{2}{n}$ and $|\Delta g_n| \leq \frac{4}{n}$.
	
So far we know that $B$ is small on the support of $g_n$, but the term $A$ also appears in $H_A$ on its own. Next we choose another vector potential $A_n$ corresponding to the same magnetic field~$B$ which is small on $B_{n}(q_n)$.
	
In local coordinates, we have $B=B_{\mu\nu} {\rm d}x^\mu \wedge {\rm d}x^\nu$ with $B_{\mu\nu}=-B_{\nu\mu}$. Then $0={\rm d}^2A={\rm d}B= \partial_\kappa B_{\mu\nu} {\rm d}x^\kappa \wedge {\rm d}x^\mu \wedge {\rm d}x^\nu$. Thus,
\[ \partial_\kappa B_{\mu\nu} - \partial_\mu B_{\kappa\nu} -\partial_\nu B_{\mu\kappa}=0.\]
	
Using the asymptotically flat coordinates on $B_{2n}(q_n)$, shifted such that $q_n$ has coordinates zero, we define on $B_{2n}(q_n)$,
\begin{align*}
A_\mu^n(x) := -\sum_{\nu>\mu} \int_0^{x^\nu} B_{\mu\nu}\bigl(x^1, \ldots, x^{\nu-1}, t, 0, \ldots, 0\bigr)  {\rm d}t.
\end{align*}	
Then
\begin{align*}
\partial_\kappa A_\mu^n(x) ={} & -\sum_{\nu>\min \{\mu, \kappa\}} \int_0^{x^\nu} \partial_\kappa B_{\mu\nu}\bigl(x^1, \ldots, x^{\nu-1}, t, 0, \ldots, 0\bigr) {\rm d}t\\& + \delta_{\mu>\kappa} B_{\mu\kappa}\bigl(x^1, \ldots, x^{\kappa}, 0, \ldots, 0\bigr). \end{align*}
Thus for $\kappa<\mu$,
\begin{align*}
\partial_\kappa A_\mu^n(x) - \partial_\mu A_\kappa^n(x)
={} & -\sum_{\nu> \kappa} \int_0^{x^\nu} \underbrace{(\partial_\kappa B_{\mu\nu} -\partial_\mu B_{\kappa\nu})}_{=\partial_\nu B_{\mu\kappa}}\bigl(x^1, \ldots, x^{\nu-1}, t, 0, \ldots, 0\bigr) {\rm d}t\\& + B_{\mu\kappa}\bigl(x^1, \ldots, x^{\kappa}, 0, \ldots, 0\bigr) \\
={} &- \sum_{\nu> \kappa} \bigl( B_{\mu\kappa}\bigl(x^1, \ldots, x^{\nu}, 0, \ldots, 0\bigr)- B_{\mu\kappa}\bigl(x^1, \ldots, x^{\nu-1}, 0, \ldots, 0\bigr) \bigr)\\&
+ B_{\mu\kappa}\bigl(x^1, \ldots, x^{\kappa}, 0, \ldots, 0\bigr) \\ ={}& -B_{\mu\kappa}(x)=B_{\kappa\mu}(x).
\end{align*}
Thus $A^n=A^n_\mu {\rm d}x^\mu$ is an equivalent potential to $A$ on $B_{2n}(q_n)$ and by construction it satisfies $|A|\leq c\frac{2}{n}$ for some uniform constant $c>0$.
		
Hence, by the Poincar\'e lemma, $A^n-A={\rm d}f_n$ for a smooth function $f_n$ on $B_{2n}(q_n)$. Let $\tilde{f}_n$ be a smooth function on all of $M$ that coincides with $f_n$ on $B_{n}(q_n)$ and set $\tilde{A}^n=A+{\rm d}\tilde{f}_n$. Then
\[ H_A= {\rm e}^{{\mathrm{i}}\tilde{f}_n} H_{\tilde{A}_n}{\rm e}^{-{\mathrm{i}}\tilde{f}_n}.\]
For $k\in {\mathbb R}^m$, we define $\phi_n= {\rm e}^{{\mathrm{i}}k_\mu x^\mu} {g}_n\Phi$. By our assumptions on $g_n$ and $\Phi$, we have
\begin{equation*}
\| \phi_n\|^2 \geq C_o \operatorname{vol}  B_{n}(q_n)
\end{equation*}
for a uniform constant $C_o$. Setting $\lambda = |k|^2$, we obtain
\begin{align*} (H_0 -\lambda)\phi_n(x) ={} & \bigl((\Delta_g -\lambda) \bigl(\tilde{g}_n{\rm e}^{{\mathrm{i}}k_\mu x^\mu}\bigr)\bigr) \Phi - 2 \nabla^S_{\nabla (\tilde{g}_n {\rm e}^{{\mathrm{i}}k_\mu x^\mu} )} \Phi + \tilde{g}_n \bigl(\nabla^S\bigr)^*\nabla^S \Phi\\
	={} & \bigl({\rm e}^{{\mathrm{i}}k_\mu x^\mu}\Delta_g \tilde{g}_n + 2\bigl\langle \nabla \tilde{g}_n, \nabla {\rm e}^{{\mathrm{i}}k_\mu x^\mu}\bigr\rangle + \tilde{g}_n (\Delta_g - |k|^2) {\rm e}^{{\mathrm{i}}k_\mu x^\mu} \bigr) \Phi\\&
- 2 \nabla^S_{\nabla (\tilde{g}_n {\rm e}^{{\mathrm{i}}k_\mu x^\mu})} \Phi + \tilde{g}_n \bigl(\nabla^S\bigr)^*\nabla^S \Phi.\end{align*}
By the decay estimates of $\Delta_g, \nabla \tilde{g}_n$ and the uniform upper and lower bound on $\Phi$ and $g_n$, we get
\begin{align*}
\bigl\|\bigl({\rm e}^{{\mathrm{i}}k_\mu x^\mu}\Delta_g \tilde{g}_n + 2\bigl\langle \nabla \tilde{g}_n, \nabla e^{{\mathrm{i}}k_\mu x^\mu}\bigr\rangle \bigr) \Phi\bigr\|^2 \leq \frac{C}{n} \operatorname{vol}  B_{2n}(q_n)\leq \frac{C}{n} \operatorname{vol}  B_{n}(q_n) \leq \frac{C}{n} \| \phi_n\|^2,
\end{align*}
where $C$ is a constant that changes from line to line but is always uniform in $n$. Similarly, the upper bounds on the derivatives of $\Phi$ give
\begin{align*}
\bigl\| \nabla^S_{\nabla  (\tilde{g}_n {\rm e}^{{\mathrm{i}}k_\mu x^\mu} )} \Phi + \tilde{g}_n \bigl(\nabla^S\bigr)^*\nabla^S \Phi \bigr\|^2
\leq \frac{C}{n^{(\tau+1)}} \operatorname{vol}  B_{2n}(q_n)\leq \frac{C}{n^{(\tau+1)}} \operatorname{vol}  B_{n}(q_n) \leq \frac{C}{n} \| \phi_n\|^2.
\end{align*}	
Moreover, \[ \bigl(\Delta_g - |k|^2\bigr) {\rm e}^{{\mathrm{i}}k_\mu x^\mu} = {\rm e}^{{\mathrm{i}}k_\mu x^\mu} \bigl(k_\mu k_\nu \bigl(g^{\mu\nu} -\delta^{\mu\nu}\bigr) -{\mathrm{i}}k_\mu g^{\mu\nu}_{,\nu} - {\mathrm{i}}k_\nu g^{\lambda \nu}\partial_\lambda (\ln \det   g)\bigr).\]
By the asymptotic flatness of $M$ we get
\[ \bigl|\bigl(\Delta_g - |k|^2\bigr) {\rm e}^{{\mathrm{i}}k_\mu x^\mu}\bigr|=\mathcal{O}\bigl(|x|^{-{\tau}}\bigr).\]
Together with the volume estimate above and the choice of $q_n$, this implies
\[\bigl\Vert \bigl(\Delta_g - |k|^2\bigr) {\rm e}^{{\mathrm{i}}k_\mu x^\mu}\bigr\Vert^2 \leq C (|q_n|-2n)^{-\tau} \operatorname{vol}  B_{2n}(q_n) \leq C n^{-\tau} \| \phi_n\|^2
\]
in a similar way as above.
Combining all of the above we get that for any $\epsilon>0$ and we can find~$n$ large enough such that{\samepage
\[
\Vert (H_0 -\lambda)\phi_n\Vert \leq \epsilon \| \phi_n\|.
\]
Set $\psi_n(x) := {\rm e}^{{\mathrm{i}}\tilde{f}_n(x)} {\rm e}^{{\mathrm{i}}k_\mu x^\mu} {g}_n(x)\Phi / \|\phi_n\|$. By construction, $\psi_n$ is an orthonormal sequence.}
	
We have
\[\bigl\Vert (H_A+1)^{-1}(H_A-\lambda)\psi_n\bigr\Vert = \bigl\Vert (H_{\tilde{A}_n}+1)^{-1}(H_{\tilde{A}_n}-\lambda)\phi_n\bigr\Vert.\]
Using $ H_{\tilde{A}_n} = H_0 -2{\mathrm{i}}  \nabla_{\tilde{A}_n} -{\mathrm{i}}  \operatorname{div} \tilde{A}_n +|\tilde{A}_n|^2 =H_0 +{\mathrm{i}}  (\nabla +{\mathrm{i}}\tilde{A}_n)^* \tilde{A}_n - {\mathrm{i}}\tilde{A}_n\nabla$,
we obtain	
\begin{align*}
\bigl\Vert (H_A+1)^{-1}&(H_A-\lambda)\psi_n\bigr\Vert
\leq\bigl\Vert (H_{\tilde{A}_n}+1)^{-1}\bigr\Vert \cdot \Vert (H_0 -\lambda)\phi_n\Vert\\
& +\bigl\Vert (H_{\tilde{A}_n}+1)^{-1} \bigl(\nabla +{\mathrm{i}}\tilde{A}_n\bigr)^*\bigr\Vert\cdot \bigl\Vert \bigl|\tilde{A}_n\bigr| \tilde{g}_n\bigl\Vert
+\bigl\Vert (H_{\tilde{A}_n}+1)^{-1} \bigr\Vert \cdot \bigl\Vert \bigl|\tilde{A}_n\bigr| \nabla (\tilde{g}_n \Phi)\bigr\Vert.
\end{align*}
Since $H_{\tilde{A}_n}$ is a nonnegative operator, by the spectral theorem $(H_{\tilde{A}_n}+1)^{-1}$ is bounded on $L^2$ uniformly in $n$. Similarly, we can show that $(H_{\tilde{A}_n}+1)^{-1} H_{\tilde{A}_n}$ is bounded on $L^2$, uniformly in $n$. Thus, also $(H_{\tilde{A}_n}+1)^{-1} (\nabla +{\mathrm{i}}\tilde{A}_n)^* $ is bounded on $L^2$ uniformly in $n$. Since $\tilde{A}_n|_{B_{2n}(q_n)}\to 0$ as $n\to \infty$ and using again the properties of $\tilde{g_n}$, $\Phi$ as before, we obtain $\bigl\Vert (H_A+1)^{-1}(H_A-\lambda)\psi_n\bigr\Vert \to 0$ and, hence, that $\lambda\in \sigma_{\mathrm{ess}} (H_A)$.
	
The remaining claim on the spectrum of $D_A$ follows from to the symmetry of the spectrum in these dimensions. The symmetry of the spectrum $D$ itself, follows for $m$ even by $\omega \cdot D\phi =-D(\omega\cdot \phi)$ by $\omega= e_1\cdot e_2 \cdots   e_{\dim M}$, and for $m\equiv 1$ mod $4$ by the existence of a Spin($m$)-equivariant real structure that anticommutes with Clifford multiplication, cf.\ \cite[Proposition, p.~31]{Fr}. Since in even dimensions $\omega\cdot$ also anticommutes with Clifford multiplication as does the real structure just mentioned, $D_A$ anticommutes with these maps as well and has, as a result, symmetric spectrum.\looseness=-1
\end{proof}

\subsection{Dense eigenvalues}
In \cite{MS}, see also \cite[Theorem~6.2]{CFKS}, an example of a magnetic potential $A$, such that the magnetic Schr\"odinger operator $(\nabla + {\mathrm{i}}A)^*(\nabla +{\mathrm{i}}A)$ on functions has pure point spectrum which is dense in $[0, \infty)$ was given. The magnetic potential used there was	
\[
A=-\frac{y}{(1+\rho)^\gamma} {\rm d}x+\frac{x}{(1+\rho)^\gamma} {\rm d}y,\]
 where $\rho=\rho(x,y)=\bigl(x^2+y^2\bigr)^\frac{1}{2}$ and $\gamma\in (0,1)$. We will prove here that this phenomenon still occurs for the Dirac operator (with the same magnetic potential). We note that \big(on ${\mathbb R}^2$\big) $D_A^*D_A =(\nabla + {\mathrm{i}}A)^*(\nabla +{\mathrm{i}}A) + {\mathrm{i}} {\rm d}A\cdot$, therefore one cannot deduce the spectrum directly from the result on the Schr\"odinger operator. But we can run an argument using similar ideas. In \cite[Theorem~6.2]{CFKS}, it is used that $(\nabla + {\mathrm{i}}A)^*(\nabla +{\mathrm{i}}A)$ commutes with the orbital angular momentum which is not true for $D_A$. But this can be circumvented by adding the spin angular momentum as we will see in the following result.

\begin{Theorem}\label{thm_dense} On $\mathbb R^2$ we consider the magnetic potential $A$ provided above.
Then, the magnetic Dirac operator $D_A= D+\mathrm{i}A$ has pure point spectrum which is dense in ${\mathbb R}$.
\end{Theorem}

\begin{proof}We consider the total angular momentum $J_z=L_z+S_{xy}$, where \smash{$L_z=-\mathrm{i}\bigl(x\frac{\partial}{\partial y}- y\frac{\partial}{\partial x}\bigr)$}
is the orbital angular momentum and $S_{xy}=\frac{\mathrm{1}}{2}\operatorname{diag}(1, -1)$ is the spin angular momentum. Direct calculation gives $[D_A, J_z]=0$ (but $L_z$ itself does not commute with $D_A$).
	
Since $\operatorname{div} A=0$, we have
\[ D_A^*D_A= \nabla^*\nabla + |A|^2 - 2{\mathrm{i}}  \nabla_A + {\mathrm{i}}  {\rm d}A(e_1,e_2)e_1\cdot e_2\cdot . \]	
Using ${\mathrm{i}}  e_1\cdot e_2 \cdot = \operatorname{diag}(1,-1)$, ${\rm d}A=B {\rm d}x\wedge {\rm d}y$ with $B= -\frac{\gamma \rho}{(1+\rho)^{\gamma +1}} + \frac{2}{(1+\rho)^{\gamma}}$ and $-{\mathrm{i}}\nabla_A =\frac{1}{(1+\rho)^{\gamma}}L_z$ this gives	
\[ D_A^*D_A = \Delta + \frac{2}{(1+\rho)^{\gamma}}J_z + \frac{\rho^2}{(1+\rho)^{2\gamma}} + \left(B -\frac{1}{(1+\rho)^{\gamma}}\right)\operatorname{diag}(1,-1).\]
The operator $J_z$ has discrete spectrum. We denote the eigenvalues by $\mu_m$, $m\in \mathbb Z$. The discreteness of the spectrum for $J_z$ allows us to write $H=L^2\bigl({\mathbb R}^2, {\mathbb C}^2\bigr)$ as the direct sum of the eigenspaces $E_m$ corresponding to the eigenvalue $\mu_m$. At the same time, the fact that $[D_A, J_z]=0$ allows one to claim that $D_A$ restricts to $E_m$. Then,
\[ D_A^*D_A|_{E_m}= \operatorname{diag}\bigl(T^+_m, T^-_m\bigr) \]
with $T^\pm_m= \Delta + \frac{2}{(1+\rho)^{\gamma}}\mu_m + \frac{\rho^2}{(1+\rho)^{2\gamma}} \pm \bigl(B -\frac{1}{(1+\rho)^{\gamma}}\bigr)$.
Thus, $T_m^\pm$ is an operator on $L^2\bigl({\mathbb R}^2, {\mathbb C}\bigr)$ of Schr\"odinger-type, in other words the Laplacian on functions plus a scalar potential. For $\gamma\in (0,1)$, the term \smash{$\frac{\rho^2}{(1+\rho)^{2\gamma}}$} is the dominating term of this potential at infinity. Thus, this potential goes to infinity at infinity and hence, $T_m^\pm$ and thus also $ D_A^*D_A|_{E_m}$ have discrete spectrum. If $D_A^*D_A$ has discrete spectrum, so does $D_A$ (see, e.g., \cite[Lemma~B.8]{AG}).\looseness=-1
	
On the other hand, $|B|\to 0$ as $|x|\to \infty$. By Theorem~\ref{thm_mag_zero}, this implies $[0,\infty)=\sigma_{\mathrm{ess}} (D_A^*D_A)$. Over ${\mathbb R}^2$, Clifford multiplication on spinors anticommutes with the map $\beta\colon (\phi_1, \phi_2)\in {\mathbb C}^2\mapsto (\overline{\phi_1}, \overline{\phi_2})$, see \cite[Section~1.7]{Fr}. Thus, $D_A(\beta(\psi))=-\beta(D_A\psi)$ which implies that the spectrum of $D_A$ is symmetric and, thus, $\sigma_{\mathrm{ess}} (D_A)={\mathbb R}$. Hence, the pure point spectrum above must be dense.
\end{proof}	

\subsection*{Acknowledgements} The authors thank Gilles Carron for asking whether for magnetic Dirac operators the spectrum can consist of a pure point spectrum that is dense in ${\mathbb R}$, which led to this article. They would also like to thank the anonymous referees whose comments helped to improve the quality of this paper.

\pdfbookmark[1]{References}{ref}
\LastPageEnding

\end{document}